\documentclass[12pt, english]{article}

\usepackage{babel,latexsym,graphicx,amsmath,amsfonts,amsthm,amssymb}
\usepackage{dsfont, fullpage, pstricks, subcaption}
\usepackage[round]{natbib}

\newtheoremstyle{localrem}
	{5pt} 
	{5pt} 
	{\rm} 
	{} 
	{\bf} 
	{{\rm.}} 
	{.7em} 
	{} 

\theoremstyle{localrem}
\newtheorem{Definition}{Definition}[section]
\newtheorem{Remark}[Definition]{Remark}
\newtheorem{Example}[Definition]{Example}

\newtheoremstyle{localthm}
	{5pt} 
	{5pt} 
	{\sl} 
	{} 
	{\bf} 
	{{\rm.}} 
	{.7em} 
	{} 

\theoremstyle{localthm}
\newtheorem{Corollary}[Definition]{Corollary}
\newtheorem{Theorem}[Definition]{Theorem}
\newtheorem{Lemma}[Definition]{Lemma}
\newtheorem{Proposition}[Definition]{Proposition}

\newcommand{\F}{\mathbb{F}}
\newcommand{\Fh}{\hat{\mathbb{F}}}

\newcommand{\N}{\mathbb{N}}

\newcommand{\R}{\mathbb{R}}

\newcommand{\XX}{\mathcal{X}}

\newcommand{\Rmua}{\R^m_{\uparrow}}
\newcommand{\Rmda}{\R^m_{\downarrow}}
\newcommand{\fundot}{\makebox[1ex]{\textbf{$\cdot$}}}

\DeclareMathOperator*{\argmin}{arg\,min}

\def\Pr{\mathop{\rm I\!P}\nolimits}

\def\bs{\boldsymbol}

\def\hat{\widehat}

\newcommand{\Xb}{\bs{X}}

\title{Monotone Least Squares and Isotonic Quantiles}
\author{Alexandre M\"{o}sching, Lutz D\"{u}mbgen \\
        University of Bern}
\date{\today}

\begin{document}
\maketitle

\begin{abstract}
We consider bivariate observations $(X_1,Y_1), \ldots, (X_n,Y_n)$ such that, conditional on the $X_i$, the $Y_i$ are independent random variables with distribution functions $F_{X_i}$, where $(F_x)_x$ is an unknown family of distribution functions. Under the sole assumption that $x \mapsto F_x$ is isotonic with respect to stochastic order, one can estimate $(F_x)_x$ in two ways:\\
(i)~For any fixed $y$ one estimates the antitonic function $x \mapsto F_x(y)$ via nonparametric monotone least squares, replacing the responses $Y_i$ with the indicators $1_{[Y_i \le y]}$.\\
(ii)~For any fixed $\beta \in (0,1)$ one estimates the isotonic quantile function $x \mapsto F_x^{-1}(\beta)$ via a nonparametric version of regression quantiles.

We show that these two approaches are closely related, with (i) being more flexible than (ii). Then, under mild regularity conditions, we establish rates of convergence for the resulting estimators $\hat{F}_x(y)$ and $\hat{F}_x^{-1}(\beta)$, uniformly over $(x,y)$ and $(x,\beta)$ in certain rectangles as well as uniformly in $y$ or $\beta$ for a fixed $x$.


\end{abstract}


\section{Introduction}

Suppose we observe $n\ge 1$ pairs
\[
	(X_1,Y_1),(X_2,Y_2),\ldots,(X_n,Y_n)\ \in \ \XX \times \R
\]
with random or fixed covariate values $X_1, \ldots, X_n$ in a set $\XX \subset \R$ such that, conditional on $\Xb = (X_i)_{i=1}^n$, the response values $Y_1, \ldots, Y_n$ are independent with
\[
	\Pr(Y_i \le y \,|\, \bs{X}) \ = \ F_{X_i}(y),
\]
for $1 \le i \le n$ and $y \in \R$. Here $(F_x)_{x \in \XX}$ is an unknown family of distribution functions on $\R$. Note that some values $X_i$ could be identical, so the corresponding random variables $Y_i$ have the same conditional distribution, given $\Xb$.

Our goal is to estimate the whole family $(F_x)_{x \in \XX}$ under the sole assumption that $x \mapsto F_x$ is isotonic (non-decreasing) with respect to stochastic order. This can be expressed in three equivalent ways:

\noindent
(SO.1) \ For arbitrary fixed $y \in \R$, \ $F_x(y)$ is antitonic (non-increasing) in $x \in \XX$.

\noindent
(SO.2) \ For any fixed $\beta \in (0,1)$, the minimal $\beta$-quantile $F_x^{-1}(\beta) := \min\{y\in\R : F_x(y) \ge \beta\}$ is isotonic in $x \in \XX$.

\noindent
(SO.3) \ For any fixed $\beta \in (0,1)$, the maximal $\beta$-quantile $F_x^{-1}(\beta\,+) := \inf\{y\in\R : F_x(y) > \beta\}$ is isotonic in $x \in \XX$.

\noindent
In what follows, we denote with $Q_x(\beta)$ any $\beta$-quantile of $F_x$ and assume that it is isotonic in $x$.

Such a constraint appears natural in several settings. For instance, an employee's income $Y$ tends to increase with his or her age $X$. Other examples in which such a stochastic order is plausible are: The expenditures $Y$ of a household for certain goods in relation to its monthly income $X$; the body height or weight $Y$ of a child in relation to its age $X$. Stochastic ordering constraints also have applications in forecasting. For example, $X_1,\ldots,X_n$ and $Y_1,\ldots,Y_n$ could be the predicted and actual cumulative precipitation amounts on $n$ different days, respectively, with the predictions being obtained from a numerical weather prediction model, see \cite{Henzi2018}.

With condition~(SO.1) in mind, one could think about estimating the antitonic function $x \mapsto F_x(y)$ by means of monotone least squares regression, replacing the response values $Y_i$ with the indicator variables $1_{[Y_i \le y]}$. Precisely, we would set $\hat{F}_x(y) = h(x)$ with an antitonic function $h : \XX \to [0,1]$ such that
\[
	\sum_{i=1}^n (1_{[Y_i \le y]} - h(X_i))^2	
\]
is minimal. The solution $h$ is unique on the set $\XX_n := \{X_1,\ldots,X_n\}$, and on $\XX \setminus \XX_n$ one could extrapolate it in some reasonable way. In the special case of $\XX$ being finite this approach has been proposed and analyzed by \citet{ElBarmi2005}.

Conditions~(SO.2-3) suggest to imitate the regression quantiles of \citet{Koenker1978}. That means, we estimate the conditional $\beta$-quantiles $Q_x(\beta)$ by $\hat{Q}_x(\beta) = h(x)$ with an isotonic function $h : \XX \to \R$ minimizing the empirical risk
\[
	\sum_{i=1}^n \rho_\beta(Y_i - h(X_i)) ,
\]
where $\rho_\beta$ denotes the loss function
\[
	\rho_\beta(z) \ := \ (\beta - 1_{[z < 0]}) z .
\]
This estimator has been considered, for instance, by \citet{PoiraudCasanova2000} who showed that it coincides with an estimator of \citet{Casady_Cryer_1976} which is given by a certain minimax formula involving sample $\beta$-quantiles. The characterization of isotonic estimators in terms of minimax formulae has also been derived by \citet{Robertson_Wright_1980} in a rather general framework including arbitrary partial orders on $\XX$ and general loss functions $R_i(\fundot)$ in place of $\rho_\beta(Y_i - \fundot)$, see also Section~\ref{subsec:Monotone.regression}.

The goals of the present paper are to clarify the connection between these two estimation paradigms and to provide new consistency results in a suitable asymptotic framework.

In Section~\ref{sec:Estimation}, we give a detailed description of the estimator $(\hat{F}_x)_{x \in \XX}$ based on monotone least squares and estimators $(\hat{Q}_x)_{x \in \XX}$ based on monotone regression quantiles. Then we show that the estimators $\hat{Q}_x$ are essentially quantiles of the estimators $\hat{F}_x$, but the latter allow for smoother estimated quantile curves.

In Section~\ref{sec:Asymptotics}, we analyze the estimators in a suitable asymptotic framework with a triangular scheme of observations and $\XX$ being a real interval. It turns out that under certain regularity conditions on the design points and the true distribution functions $F_x$, one can prove rates of convergence for quantities such as
\[
	\sup_{x \in I, y \in J} \, \bigl| \hat{F}_x(y) - F_x(y) \bigr|
	\quad\text{and}\quad
	\sup_{x \in I, \beta \in B} \, \bigl| \hat{Q}_x(\beta) - Q_x(\beta) \bigr|
\]
with intervals $I \subset \XX$, $J \subset \R$ and $B \subset (0,1)$. These results generalize and improve the findings of \cite{Casady_Cryer_1976}, see also \cite{Mukerjee_1993} who analyzed a slightly different estimator. In addition we investigate
\[
	\sup_{y \in J} \, \bigl| \hat{F}_{x_o}(y) - F_{x_o}(y) \bigr|
	\quad\text{and}\quad
	\sup_{\beta \in B} \, \bigl| \hat{Q}_{x_o}(\beta) - Q_{x_x}(\beta) \bigr|
\]
for a fixed interior point $x_o$ of $\XX$. These results complement the analysis of a single quantile curve by \cite{Wright_1984}.

Proofs and technical details are deferred to Section~\ref{sec:Proofs}. We also provide some general results about isotonic regression which are of independent interest.

\section{Estimation of the conditional distributions}
\label{sec:Estimation}

Throughout this section, we view the observations $(X_i,Y_i)$, $1 \le i \le n$, as fixed and focus mainly on computational aspects. Let $x_1 < \cdots < x_m$ be the different elements of the set $\XX_n$ of observed values $X_i$, that means, $m \le n$. For $1 \le j \le m$, we set
\[
	w_j \ := \ \# \{i : X_i = x_j\} .
\]
Then
\[
	\Pr(Y_i \le y) \ = \ F_{x_j}(y)
	\quad\text{whenever} \ X_i = x_j ,
\]
and the unconstrained maximum likelihood estimator of $F_{x_j}(y)$ is given by
\begin{equation}
\label{eq:Fhat.naive}
	\Fh_j(y) \ := \ w_j^{-1} \sum_{i \,:\, X_i = x_j} 1_{[Y_i \le y]} .
\end{equation}

\subsection{Estimation of $F_x$ via monotone least squares}
\label{subsec:Fx.LS}

The estimator $\Fh_j(y)$ in \eqref{eq:Fhat.naive} is rather poor by itself, unless the corresponding subsample size $w_j$ is large. But in connection with our stochastic order constraint, it becomes a useful tool. Note first that, for any function $h : \XX \to \R$,
\[
	\sum_{i=1}^n (1_{[Y_i \le y]} - h(X_i))^2
	\ = \ \sum_{j=1}^m w_j \bigl( \Fh_j(y) - h(x_j) \bigr)^2
		+ \sum_{j=1}^m w_j \Fh_j(y)\bigl(1 - \Fh_j(y)\bigr) , 
\]
and the stochastic order assumption implies that the vector $\bs{F}(y) = (F_{x_j}(y))_{j=1}^m$ belongs to the cone
\[
	\Rmda \ := \ \{\bs{f} \in \R^m : f_1 \ge f_2 \ge \cdots \ge f_m\} . 
\]
Hence one can estimate $\bs{F}(y)$ by the unique least squares estimator
\[
	\hat{\bs{F}}(y) = \bigl( \hat{F}_{x_j}(y) \bigr)_{j=1}^m
	\ := \ \argmin_{\bs{f} \in \Rmda} \, \sum_{j=1}^m w_j \bigl( \Fh_j(y) - f_j \bigr)^2 .
\]

It is well-known that $\hat{\bs{F}}(y)$ may also be represented by the following minimax and maximin formulae, see \citet{Robertson1988}: For $1 \le j \le m$,
\begin{equation}
\label{eq:minmaxmin.LS}
	\hat{F}_{x_j}(y)
	\ = \ \min_{r\leq j} \max_{s\geq j}\, \Fh_{rs}(y)
	\ = \ \max_{s\geq j} \min_{r\leq j}\, \Fh_{rs}(y) ,
\end{equation}
where
\begin{align*}
	\Fh_{rs}(y) \
	&:= \ w_{rs}^{-1}\sum_{j=r}^s w_j \Fh_j(y)
		\ = \ \argmin_{f \in \R} \, \sum_{j=r}^s w_j \bigl( \Fh_j(y) - f \bigr)^2 , \\
	w_{rs} \
	&:= \ \sum_{j=r}^s w_j \ = \ \# \{i:x_r\leq X_i\leq x_s\} ,
\end{align*}
and $r,s$ stand for indices in $\{1,2,\ldots,m\}$ such that $r\leq s$. These formulae are useful for theoretical considerations. In particular, since the pointwise maximum or minimum of finitely many distribution functions is a distribution function, too, we may conclude that for $1 \le j \le m$,
\[
	\hat{F}_{x_j}(\cdot)
	\ \text{is a distribution function} .
\]

The computation of $\hat{\bs{F}}(y)$ is easily accomplished via the pool-adjacent-violators algorithm (PAVA), see \citet{Robertson1988}. Note also that it suffices to compute $\hat{\bs F}(y)$ for at most $n - 1$ different values of $y$. Precisely, if $y_1 < y_2 < \cdots < y_\ell$ are the elements of $\{Y_1,Y_2,\ldots,Y_n\}$, then $\hat{\bs{F}}(y) = \bs{0}$ for $y < y_1$, $\hat{\bs{F}}(y) = \bs{1}$ for $y \ge y_\ell$, and $\hat{\bs{F}}(y) = \hat{\bs{F}}(y_k)$ for $1 \le k < \ell$ and $y \in [y_k,y_{k+1})$. Consequently, since the PAVA is known to have linear complexity, the computation of all estimators $\hat{F}_{x_j}(\cdot)$, $1 \le j \le m$, requires $O(n \log n + m\ell) = O(n^2)$ steps.

Finally, we extrapolate $\hat{\bs{F}}(y)$ to an antitonic function $x \mapsto \hat{F}_x(y)$ on $\XX$. We set $\hat{F}_x(y) := \hat{F}_{x_1}(y)$ for $x \le x_1$ and $\hat{F}_x(y) := \hat{F}_{x_m}(y)$ for $x \ge x_m$. For $x_{j-1} \le x \le x_j$, $1 < j \le m$, one could define $\hat{F}_x(y)$ by linear interpolation, but other antitonic interpolations are possible without affecting our asymptotic results.

\subsection{Plug-in estimation of $Q_x$}
\label{subsec:Qx.plug-in}

Once we have estimated $(F_x)_{x\in\XX}$ by $(\hat{F}_x)_{x\in\XX}$ as in Section~\ref{subsec:Fx.LS}, we can easily determine corresponding quantile functions. For any fixed $\beta \in (0,1)$ and $x_j$, $1\leq j\leq m$, we could determine the minimal and maximal $\beta$-quantiles,
\[
	\hat{F}_{x_j}^{-1}(\beta) \ := \ \min \bigl\{ y \in \R : \hat{F}_{x_j}(y) \ge \beta \bigr\}
	\quad\text{and}\quad
	\hat{F}_{x_j}^{-1}(\beta\,+) \ := \ \inf \bigl\{ y \in \R : \hat{F}_{x_j}(y) > \beta \bigr\} .
\]
Both vectors $(\hat{F}_{x_j}^{-1}(\beta))_{j=1}^m$ and $(\hat{F}_{x_j}^{-1}(\beta\,+))_{j=1}^m$ are isotonic, and any choice of an isotonic function $\XX \ni x \mapsto \hat{Q}_x(\beta)$ such that $\hat{F}_{x_j}^{-1}(\beta) \le \hat{Q}_{x_j}(\beta) \le \hat{F}_{x_j}^{-1}(\beta\,+)$, $1\leq j \leq m$, is a plausible estimator of a $\beta$-quantile curve.

\subsection{Estimation of $Q_x$ via monotone regression quantiles}
\label{subsec:Qx.RQ}

Similarly as in Section~\ref{subsec:Fx.LS}, we focus on the vector $\bs{Q}(\beta) = ( Q_{x_j}(\beta) )_{j=1}^m$. Writing
\[
	\sum_{i=1}^n \rho_\beta(Y_i - h(X_i))
	\ = \ \sum_{j=1}^m \sum_{i:\,X_i=x_j} \rho_\beta(Y_i - h(x_j)) ,
\]
one can estimate $\bs{Q}(\beta)$ by some vector in the set
\[
	\hat{\mathcal{Q}}(\beta) \
	:= \ \argmin_{\bs{q}\in\Rmua} \, T_\beta(\bs{q}),
\]
where $\Rmua := -\Rmda = \{\bs{q} \in \R^m : q_1 \le q_2 \le \cdots \le q_m\}$ and
\[
	T_\beta(\bs{q}) \ := \ \sum_{j=1}^m \sum_{i:\,X_i=x_j} \rho_{\beta}(Y_i - q_j) .
\]
Note that the function $T_\beta(\fundot)$ is convex but not strictly convex on $\R^m$. Hence it need not have a unique minimizer. The next result provides more precise information in terms of the minimal and maximal sample $\beta$-quantiles
\begin{align*}
	\Fh_{rs}^{-1}(\beta) \
	&:= \ \min \bigl\{ y \in \R : \Fh_{rs}(y) \ge \beta \bigr\} , \\
	\Fh_{rs}^{-1}(\beta\,+) \
	&:= \ \inf \bigl\{ y \in \R : \Fh_{rs}(y) > \beta \bigr\}.
\end{align*}

\begin{Lemma}
\label{lem:Qhat.RQ}
The set $\hat{\mathcal{Q}}(\beta)$ is a compact and convex subset of $\Rmua$.\\[0.5ex]
Two particular elements of $\hat{\mathcal{Q}}(\beta)$ are the vectors $\bs{\ell} = (\ell_j)_{j=1}^m$ and $\bs{u} = (u_j)_{j=1}^m$ with components
\begin{align*}
	\ell_j \
	&:= \ \max_{r \le j} \, \min_{s \ge j} \, \Fh_{rs}^{-1}(\beta)
		\ = \ \min_{s \ge j} \, \max_{r \le j} \, \Fh_{rs}^{-1}(\beta) , \\
	u_j \
	&:= \ \min_{s \ge j} \, \max_{r \le j} \, \Fh_{rs}^{-1}(\beta\,+)
		\ = \ \max_{r \le j} \, \min_{s \ge j} \, \Fh_{rs}^{-1}(\beta\,+) .
\end{align*}
Any vector $\bs{q} \in \hat{\mathcal{Q}}(\beta)$ satisfies $\bs{\ell} \le \bs{q} \le \bs{u}$ componentwise.\\[0.5ex]
On the other hand, suppose that $\bs{q} \in \Rmua$ satisfies $\bs{\ell} \le \bs{q} \le \bs{u}$ and that $\{j < m : q_j < q_{j+1}\}$ is a subset of $\{j < m : \ell_j < \ell_{j+1} \ \text{or} \ u_j < u_{j+1}\}$. Then $\bs{q} \in \hat{\mathcal{Q}}(\beta)$.\\[0.5ex]
Finally, for any $j \in \{1,\ldots,m\}$, the set $\{x_j\} \times (\ell_j,u_j)$ contains no data point $(X_i,Y_i)$.
\end{Lemma}

\begin{Remark}
\label{rem:n=2}
At first glance, one might suspect that any isotonic vector $\bs{q} \in \Rmua$ satisfying $\bs{\ell} \le \bs{q} \le \bs{u}$ minimizes $T_\beta$. But this conjecture is wrong. As a counterexample, consider the case of $n = 2$ observations with $X_1 < X_2$ but $Y_1 > Y_2$. Here $m = 2$, and $\Fh_{11}(y) = 1_{[y \ge Y_1]}$, $\Fh_{22}(y) = 1_{[y \ge Y_2]}$ and 
\[
	\Fh_{12}(y) \ = \ \begin{cases}
		0 & \text{if} \ y < Y_2, \\
		0.5 & \text{if} \ Y_2 \le y < Y_1 , \\
		1 & \text{if} \ y \ge Y_1 .
	\end{cases}
\]
Hence
\[
	\bs{\ell} \ = \ (Y_2,Y_2)^\top
	\quad\text{and}\quad
	\bs{u} \ = \ (Y_1,Y_1)^\top ,
\]
because $\Fh_{11}^{-1}(0.5) = \Fh_{11}^{-1}(0.5\,+) = Y_1$, $\Fh_{22}^{-1}(0.5) = \Fh_{22}^{-1}(0.5\,+) = Y_2$ and
\[
	\Fh_{12}^{-1}(0.5) \ = \ Y_2, \quad
	\Fh_{12}^{-1}(0.5\,+) \ = \ Y_1 .
\]
But
\[
	\hat{\mathcal{Q}}(0.5) \ = \ \bigl\{ (q,q)^\top : q \in [Y_2,Y_1] \bigr\} ,
\]
because for $\bs{q} \in [Y_2,Y_1]^2$ with $q_1 \le q_2$,
\[
	\rho_{0.5}(Y_1 - q_1) + \rho_{0.5}(Y_2 - q_2)
	\ = \ 0.5 (Y_1 - q_1 + q_2 - Y_2)
	\ \ge \ 0.5(Y_1 - Y_2)
\]
with equality if, and only if, $q_1 = q_2$.
\end{Remark}

\subsection{Connection between the two estimation paradigms}

Restricting the plug-in quantile estimators of Section~\ref{subsec:Qx.plug-in} to the set $\XX_n$ of observed $X$-values leads to the set
\[
	\hat{\mathcal{Q}}_{\mathrm{plug-in}}(\beta) \
	:= \ \bigl\{ \bs{q}\in\Rmua :
		\hat{F}_{x_j}^{-1}(\beta) \le q_j \le \hat{F}_{x_j}^{-1}(\beta\,+)
		\ \text{for} \ 1 \le j \le m \bigr\} .
\]
This set is closely related to the set $\hat{\mathcal{Q}}(\beta)$:

\begin{Lemma}
\label{lem:Qhat.LS}
The vectors $\bs{\ell}$ and $\bs{u}$ in Lemma~\ref{lem:Qhat.RQ} are given by
\[
	\ell_j \ = \ \hat{F}_{x_j}^{-1}(\beta)
	\quad\text{and}\quad
	u_j \ = \ \hat{F}_{x_j}^{-1}(\beta\,+)
	\quad\text{for} \ 1 \le j \le m .
\]
In particular, $\hat{\mathcal{Q}}(\beta) \subset \hat{\mathcal{Q}}_{\mathrm{plug-in}}(\beta)$.
\end{Lemma}

\begin{Example}
The simple example in Remark~\ref{rem:n=2} shows that $\hat{\mathcal{Q}}(\beta) \ne \hat{\mathcal{Q}}_{\rm plug-in}(\beta)$ in general. Let us illustrate this point with a more interesting numerical example. Figure~\ref{Fig:Median} shows a simulated sample of size $n = 100$. In addition, it shows the minimal and maximal median curves $x \mapsto \hat{F}_x^{-1}(0.5), \hat{F}_{x}^{-1}(0.5\,+)$ obtained by linear interpolation of the points $\ell_j = \hat{F}_{x_j}^{-1}(0.5)$ and $u_j = \hat{F}_{x_j}^{-1}(0.5\,+)$, respectively, as well as a piecewise linear median curve $x \mapsto \hat{Q}_x(0.5)$ minimizing $\int q'(x)^2 \, dx$ among all isotonic functions $q : \R \to \R$ such that $\ell_j \le q(x_j) \le u_j$, $1\leq j\leq m$. Although $\hat{Q}_x(0.5)$ is a natural candidate and smoother in $x$ than $\hat{F}_x^{-1}(0.5)$ or $\hat{F}_x^{-1}(0.5\,+)$, the corresponding values of $T_{0,5}(\fundot)$ are (rounded to three digits)
\[
	T_{0.5} \Bigl( \bigl(\hat{Q}_{x_j}(0.5) \bigr)_{j=1}^m \Bigr)
	\ = \ 45.343 \ > \ T_{0.5}(\bs{\ell}) = T_{0,5}(\bs{u}) \ = \ 44.112 .
\]
The true medians $F_x^{-1}(0.5) = F_x^{-1}(0.5\,+)$ are depicted as well.

\begin{figure}[t]
\centering
\includegraphics[width=0.99\textwidth]{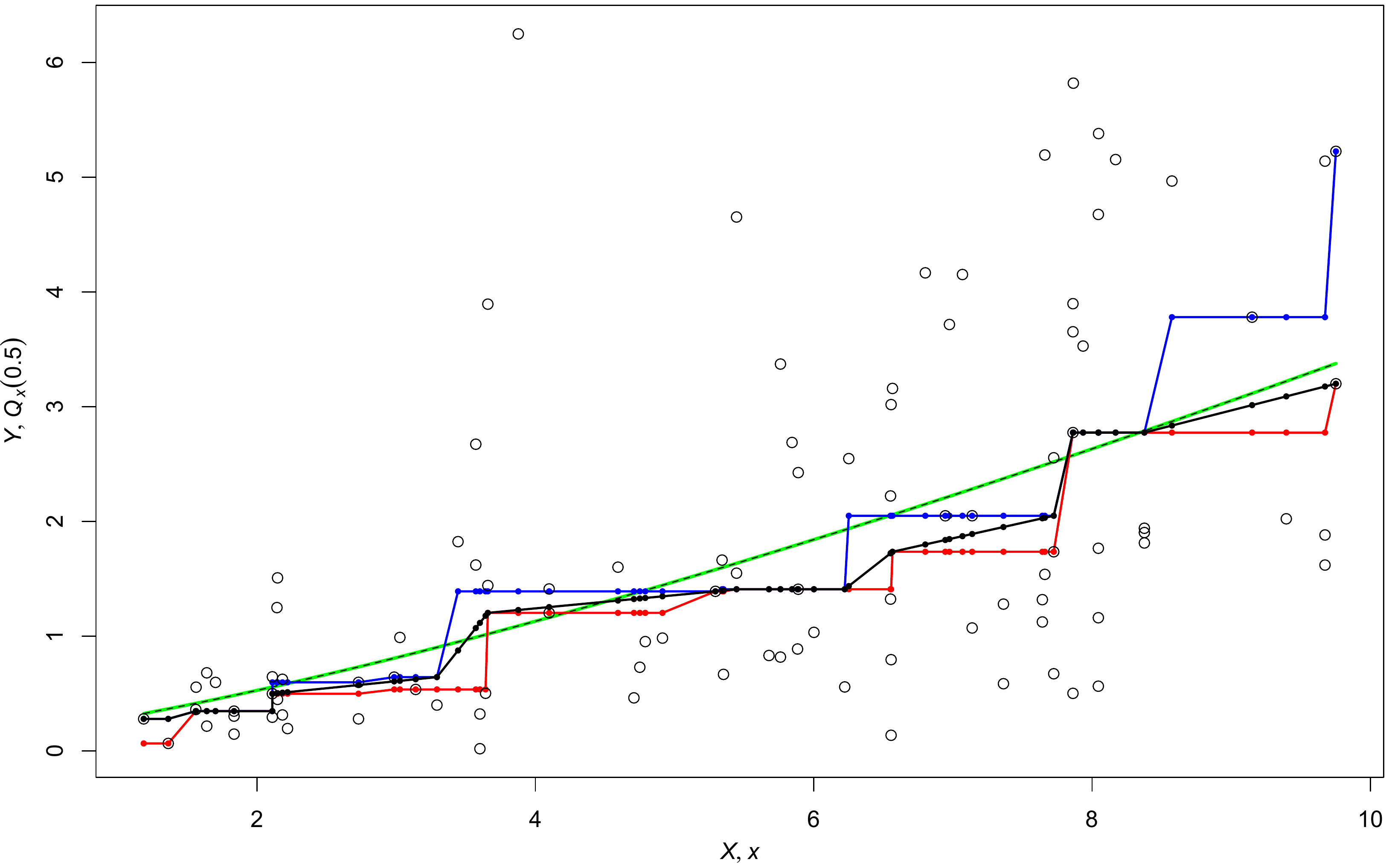}
\caption{\label{Fig:Median} $n = 100$ data pairs, together with the true medians $F_x^{-1}(0.5)$ (green, dashed) and the estimated medians $\hat{F}_x^{-1}(0.5)$ (lower red), $\hat{F}_x^{-1}(0.5\,+)$ (upper blue) and $\hat{Q}_x(0.5)$ (middle black).}
\end{figure}
\end{Example}

\section{Asymptotic considerations}
\label{sec:Asymptotics}

We provide some asymptotic properties of the estimators just introduced in case of a real interval $\XX$ and a triangular scheme of observations: For each sample size $n \ge 2$, consider observations $(X_{n1},Y_{n1}), \ldots, (X_{nn}, Y_{nn})$ with $X_{n1},\ldots,X_{nn} \in \XX$ such that conditional on $\Xb_n := (X_{ni})_{i=1}^n$, the random variables $Y_{n1},\ldots,Y_{nn}$ are independent with
\[
	\Pr(Y_{ni} \le y \,|\, \Xb_n) \ = \ F_{X_{ni}}^{}(y),
\]
for $1 \le i \le n$ and $y \in \R$. The resulting constrained estimators of $F_x(y)$ and $Q_x(\beta)$ are denoted by $\hat{F}_{nx}(y)$ and $\hat{Q}_{nx}(\beta)$, respectively. In what follows, we derive asymptotic properties of these estimators under moderate assumptions, where asymptotic statements refer to $n \to \infty$.

\cite{ElBarmi2005} derived asymptotic properties in case of a fixed finite set~$\XX$, which is easier to handle than the present setting.

\subsection{Uniform consistency in both arguments}

First of all, we assume that the distribution functions $F_x$ are H\"older-continuous in $x$, at least on some subinterval of $\XX$:

\paragraph{(A.1)}
For given intervals $I \subset \XX$ and $J \subset \R$, there exist constants $\alpha \in (0,1]$ and $C_1 > 0$ such that
\[
	\sup_{y\in J} \, \bigl| F_w(y) - F_x(y) \bigr|
	\ \le \ C_1 |w - x|^\alpha
	\quad\text{for arbitrary} \ w,x \in I.
\]

Secondly, we assume that the design points are `asymptotically dense' within this interval $I$. To state this precisely, we need some notation. We write
\[
	\rho_n \ := \ \frac{\log n}{n} ,
\]
and $\lambda(\fundot)$ stands for Lebesgue measure. Moreover, the absolute frequency of the design points $X_{ni}$ is denoted by $w_n(\fundot)$, that means,
\[
	w_n(B) \ := \ \#\{i \le n : X_{ni} \in B\}
	\quad\text{for} \ B \subset \XX .
\]

\paragraph{(A.2)}
For given constants $C_2, C_3 > 0$, let $A_n$ be the event that for arbitrary intervals $I_n \subset I$,
\[
	\frac{w_n(I_n)}{n\lambda(I_n)} \ \ge \ C_2
	\quad\text{whenever}\quad
	\lambda(I_n) \ \ge \ \delta_n := C_3 \rho_n^{1/(2\alpha+1)} .
\]
Then,
\[
	\Pr(A_n) \ \to \ 1 .
\]

\begin{Remark}[Fixed design points]
Suppose that $I = \XX = [a,b]$ with real numbers $a < b$, and let $X_{ni} = a + (i/n)(b-a)$ for $1 \le i \le n$. Then Assumption~(A.2) is satisfied for any fixed $C_2 < 1$ and $C_3 > 0$.
\end{Remark}

\begin{Remark}[Random design points]
Suppose that $X_{n1}, X_{n2}, \ldots, X_{nn}$ are independent random variables with density $g$ on $\XX$ such that $\inf_{x \in I} g(x) > 0$ on $I$. With standard results from empirical processes on the real line, including exponential inequalities for beta distributions, we can show that for any choice of $\alpha \in (0,1]$, $0 < C_2 < \inf_{x \in I} g(x)$ and $C_3 > 0$,
\[
	\inf \Bigl\{ \frac{w_n(I_n)}{n\lambda(I_n)} :
		\text{intervals} \ I_n \subset I \ \text{with} \ \lambda(I_n) \ge \delta_n \Bigr\}
	\ \ge \ C_2
\]
with asymptotic probability one as $n \to \infty$. Hence Assumption~(A.2) is satisfied.
\end{Remark}

Under the two assumptions above, the estimator $\hat{F}_{nx}$ satisfies a uniform consistency property.

\begin{Theorem}
\label{Thm:F_Asymptotics}
Suppose that Assumptions~(A.1--2) are satisfied. Then there exists a $C = C(C_1,C_2,C_3) > 0$ such that
\[
	\lim_{n\to \infty}	
	\Pr\Bigl(
		\sup_{x \in I_n,y\in J} \, \bigl| \hat{F}_{nx}(y) - F_x(y) \bigr|
		\ \geq \ 
		C \rho_n^{\alpha/(2\alpha+1)} 
	\Bigr)
	\ = \ 0,
\]
where $I_n := \{x \in \R : [x \pm \delta_n] \subset I\}$.
\end{Theorem}

Concerning estimated quantiles, we combine Assumptions~(A.1--2) with a growth condition on the conditional distribution functions $F_x$:

\paragraph{(A.3)}
For some numbers $0\leq \beta_1 < \beta_2 \leq 1$ and $\kappa>0$,
\[
	F_x(y_2)-F_x(y_1) \ \geq \ \kappa(y_2-y_1),
\]
for arbitrary $x\in I$ and $y_1,y_2\in \R$ such that $y_1 < y_2$ and $F_x(y_1), F_x(y_2-) \in (\beta_1, \beta_2)$.

For instance, if each $F_x$, $x \in I$, has a density $f_x$ such that
\[
	\kappa := \inf_{x \in I} \ \inf_{y \, : \, \beta_1 < F_x(y) < \beta_2} \, f_x(y)
	\ > \ 0 ,
\]
then (A.3) is satisfied with the latter parameter $\kappa$.

\begin{Theorem}
\label{Thm:Q_Asymptotics}
Suppose that Assumptions~(A.1--3) are satisfied with $J = \R$ in (A.1). Then, for any plug-in estimator $(\hat{Q}_{nx})_{x \in \XX}$ of $(Q_x)_{x \in \XX}$,
\[
	\lim_{n\to \infty}	
	\Pr\Bigl(
		\sup_{x\in I_n,\beta\in B_n} \, 
			\bigl| \hat{Q}_{nx}(\beta) - Q_x(\beta) \bigr|
		\ > \ 
		\kappa^{-1}C \rho_n^{\alpha/(2\alpha+1)} 
	\Bigr)
	\ = \ 0,
\]
where $I_n \subset I$ and $C = C(C_1,C_2,C_3)$ are defined as in Theorem~\ref{Thm:F_Asymptotics}, and $B_n$ denotes the interval $(\beta_1 + C\rho_n^{\alpha/(2\alpha + 1)}, \beta_2 - C\rho_n^{\alpha/(2\alpha + 1)})$.
\end{Theorem}

\subsection{Uniform consistency at a single point $x_o$}

In addition to the previous uniform convergence results, one may verify uniform consistency of $\hat{F}_{nx_o}$ and $\hat{Q}_{nx_o}$ for a fixed interior point $x_o$ of $\XX$. These results require similar but weaker assumptions.

\paragraph{(A'.1$\bs{_{x_o}}$)}
For a neighbourhood $U\subset \mathcal{X}$ of $x_o$ and an interval $J \subset \R$, there exist constants $\alpha \in (0,1]$ and $C_1 > 0$ such that
\[
	\sup_{y\in J} \, \bigl| F_x(y) - F_{x_o}(y) \bigr|
	\ \le \ C_1 |x - x_o|^\alpha
	\quad\text{for arbitrary} \ x \in U.
\]

\paragraph{(A'.2$\bs{_{x_o}}$)}
For given constants $C_2, C_3 > 0$, let $A_n$ be the event that
\[
	\frac{w_n([x_o-\delta_n,x_o])}{n \delta_n},
	\frac{w_n([x_o,x_o+\delta_n])}{n \delta_n}
	\ \ge \ C_2
	\quad\text{where}\quad
	\delta_n := C_3 n^{-1/(2\alpha+1)} .
\]
Then,
\[
	\Pr(A_n) \ \to \ 1 .
\]

Under these two assumptions, the following consistency property holds.

\begin{Theorem}
\label{Thm:F_Asymptotics_xo}
Suppose that Assumptions~(A'.1--2$_{x_o}$) are satisfied. Then
\[
	\sup_{y\in J} \, \bigl| \hat{F}_{nx_o}(y) - F_{x_o}(y) \bigr|
	\ = \ O_p \bigl( n^{-\alpha/(2\alpha + 1)} \bigr).
\]
\end{Theorem}

\paragraph{(A'.3$\bs{_{x_o}}$)}
For some numbers $0\leq \beta_1 < \beta_2 \leq 1$ and $\kappa>0$,
\[
	F_{x_o}(y_2)-F_{x_o}(y_1) \ \geq \ \kappa(y_2-y_1),
\]
for arbitrary $y_1,y_2\in \R$ such that $y_1 < y_2$ and $F_{x_o}(y_1), F_{x_o}(y_2-) \in (\beta_1, \beta_2)$.

\begin{Theorem}
\label{Thm:Q_Asymptotics_xo}
Suppose that Assumptions~(A.1--3$_{x_o}$) are satisfied with $J = \R$ in (A.1$_{x_o}$). Then, for any plug-in estimator $(\hat{Q}_{nx})_{x \in \XX}$ of $(Q_x)_{x \in \XX}$,
\[
	\sup_{\beta \in B_n} \,
		\bigl| \hat{Q}_{nx_o}(\beta) - Q_{x_o}(\beta) \bigr|
	\ = \ O_p \bigl( n^{-\alpha/(2\alpha + 1)} \bigr),
\]
where $B_n := (\beta_1 + \Delta_n, \beta_2 - \Delta_n)$ and $\Delta_n=\mathcal{O}(n^{-\alpha/(2\alpha+1)})$.
\end{Theorem}

\section{Proofs and technical details}
\label{sec:Proofs}

\subsection{Monotone regression}
\label{subsec:Monotone.regression}

In this section we review isotonic regression on a totally ordered set in a rather general setting, summarizing and extending results of numerous authors. Our main goal is a thorough understanding of isotonic regression in situations with potentially non-unique solutions. For extensions to partially ordered sets we refer to \cite{Muehlemann2019}.

The starting point are $m \ge 2$ loss functions $R_1, \ldots, R_m : \R \to \R$ with the following property: For arbitrary indices $1 \le a \le b \le m$, the function
\[
	R_{ab} \ := \ \sum_{j=a}^b R_j
\]
is minimal on a compact interval $[L_{ab}, U_{ab}] \subset \R$, strictly antitonic on $(-\infty,L_{ab}]$ and strictly isotonic on $[U_{ab},\infty)$.

This property is satisfied if all functions $R_j$ are convex with $R_j(x) \to \infty$ as $|x| \to \infty$. It implies a refined version of the so-called Cauchy-mean-value property.

\begin{Proposition}
\label{Prop:Cauchy.mean.value}
Let $\{a,\ldots,b\} \subset \{1,\ldots,m\}$ be partitioned into $k \ge 2$ index intervals $\{a_1,\ldots,b_1\}, \ldots, \{a_k,\ldots,b_k\}$. Then
\[
	\min_{1 \le i \le k} L_{a_ib_i}
	\ \le \ L_{ab} \ \le \ \max_{1 \le i \le k} L_{a_ib_i}
	\quad\text{and}\quad
	\min_{1 \le i \le k} U_{a_ib_i}
	\ \le \ U_{ab} \ \le \ \max_{1 \le i \le k} U_{a_ib_i} .
\]
\end{Proposition}

\begin{proof}[\textbf{Proof}]
The smallest minimizer $L_{ab}$ of $R_{ab}$ is the largest real number $r$ such that $R_{ab}$ is strictly antitonic on $(-\infty,r]$ and the smallest real number $s$ such that $R_{ab}$ is isotonic on $[s,\infty)$. Since $R_{ab} = \sum_{i=1}^k R_{a_ib_i}$, this function is strictly antitonic on $\bigcap_{1\le i \le k} (-\infty,L_{a_ib_i}] = \bigl( - \infty, \min_{1 \le i \le k} L_{a_ib_i} \bigr]$ and isotonic on $\bigcap_{1\le i \le k} [L_{a_ib_i}, \infty) = \bigl[ \max_{1 \le i \le k} L_{a_ib_i}, \infty \bigr)$. This yields the desired inequalities for $L_{ab}$. The largest minimizer $U_{ab}$ can be handled analogously.
\end{proof}

Now we consider the function $T : \R^m \to \R$,
\[
	T(\bs{x}) \ := \ \sum_{j=1}^m R_j(x_j)
\]
and the set
\[
	\mathcal{Q} \ := \ \argmin_{\bs q\in\Rmua} \, T(\bs{q}) .
\]
The elements of $\mathcal{Q}$ can be characterized completely in terms of the minimizers of the functions $R_{ab}$. Throughout the sequel, we set $x_0 := -\infty$ and $x_{m+1} := \infty$ for a vector $\bs{x} \in \Rmua$. Moreover, the componentwise minimum and maximum of vectors $\bs{x}, \bs{y} \in \R^m$ are denoted by $\min(\bs{x},\bs{y})$ and $\max(\bs{x},\bs{y})$, respectively.

\begin{Proposition}
\label{Prop:Characterization}
For a vector $\bs{x} \in \Rmua$, the following two properties are equivalent:

\noindent
\textbf{(i)} \ $\bs{x} \in \mathcal{Q}$.

\noindent
\textbf{(ii)} \ For arbitrary indices $1 \le a \le b \le m$,
\begin{align*}
	x_a \ \le \ U_{ab} &\quad\text{if} \ x_{a-1} < x_a , \\
	x_b \ \ge \ L_{ab} &\quad\text{if} \ x_b < x_{b+1} .
\end{align*}
\end{Proposition}

This characterization is a generalization of Theorem~8.1 of \cite{Duembgen2009}.

\begin{proof}[\textbf{Proof of Proposition~\ref{Prop:Characterization}}]
We first show that property~(i) is equivalent to a seemingly weaker version of (ii):

\noindent
\textbf{(ii')} \ For arbitrary indices $1 \le a \le b \le m$,
\begin{align*}
	x_a \ \le \ U_{ab} &\quad\text{if} \ x_{a-1} < x_a = x_b , \\
	x_b \ \ge \ L_{ab} &\quad\text{if} \ x_a = x_b < x_{b+1} .
\end{align*}

Suppose that property~(ii') is violated. Specifically, for some indices $1 \le a \le b \le m$, let $x_{a-1} < x_a = x_b$ but $x_a > U_{ab}$. Since $R_{ab}$ is strictly isotonic on $[U_{ab},\infty)$,
\[
	\tilde{x}_j \ := \ \begin{cases}
		x_j & \text{if} \ j < a \ \text{or} \ j > b \\
		\max(x_{a-1}, U_{ab}) & \text{if} \ a \le j \le b
	\end{cases}
\]
defines a vector $\tilde{\bs{x}} \in \Rmua$ such that $T(\tilde{\bs{x}}) < T(\bs{x})$. Analogously, if $x_a = x_b < x_{b+1}$ but $x_b < L_{ab}$, one can find a vector $\tilde{\bs{x}} \in \Rmua$ such that $T(\tilde{\bs{x}}) < T(\bs{x})$. This shows that property~(i) implies property~(ii').

Suppose that property~(ii') is satisfied, and let $\bs{y}$ be an arbitrary vector in $\Rmua$. If $y_j > x_j$ for some index $j$, let $a$ be the smallest such index, and let $c$ be the largest index with $x_c = x_a$. Thus $x_a = x_c < x_{c+1}$ and $y_{a-1} \le x_a < y_a \le y_c$. Now we repeat the following step until $y_c = x_c$: We choose the smallest index $b$ such that $y_b = y_c$. Property~(ii') implies that $x_c \ge L_{bc}$, so $R_{bc}$ is isotonic on $[x_c,\infty)$. Consequently, if we replace $y_b,\ldots,y_c$ with the smaller number $\max(x_c,y_{b-1})$, the value $T(\bs{y})$ does not increase. These considerations show that replacing $y_a,\ldots,y_c$ with $x_a = x_c$ yields a new vector $\bs{y} \in \Rmua$ with the same or a smaller value of $T(\bs{y})$. Repeating this construction finitely often shows that replacing $\bs{y}$ with $\min(\bs{x},\bs{y})$ does not increase $T(\bs{y})$. Analogously one can show that replacing $\bs{y}$ with $\max(\bs{x},\bs{y})$ does not increase $T(\bs{y})$. Combining both steps shows that the original vector $\bs{y}$ satisfies the inequality $T(\bs{y}) \ge T(\bs{x})$. Hence $\bs{x}$ belongs to $\mathcal{Q}$.

It remains to show equivalence of properties~(ii) and (ii'). The latter is obviously a consequence of the former one. Hence it suffices to show that a violation of property~(ii) implies a violation of (ii'). Consider indices $1 \le a \le b \le m$ such that $x_{a-1} < x_a$ but $x_a > U_{ab}$. In case of $x_b = x_a$, this is a violation of property~(ii). In case of $x_a < x_b$ we partition $\{a,\ldots,b\}$ into maximal index intervals $\{a_1,\ldots,b_1\}, \ldots, \{a_k,\ldots,b_k\}$ on which $j \mapsto x_j$ is constant. Then $x_a = \min_{1 \le i \le k} x_{a_i}$, whereas Proposition~\ref{Prop:Cauchy.mean.value} yields the inequality $U_{ab} \ge \min_{1 \le i \le k} U_{a_ib_i}$. Hence for some index $i$, $x_{a_i-1} < x_{a_i} = x_{b_i}$ but $x_{a_i} > U_{a_ib_i}$, a violation of (ii). The situation that $x_b < x_{b-1}$ but $x_b < L_{ab}$ can be handled analogously.
\end{proof}

Proposition~\ref{Prop:Characterization} implies already an interesing property of the set $\mathcal{Q}$.

\begin{Corollary}
\label{Cor:Pointwise.min.max}
If $\bs{x}^{(1)}, \bs{x}^{(2)} \in \mathcal{Q}$, then $\min(\bs{x}^{(1)},\bs{x}^{(2)})$ and $\max(\bs{x}^{(1)},\bs{x}^{(2)})$ belong to $\mathcal{Q}$ as well.
\end{Corollary}

\begin{proof}[\textbf{Proof}]
For symmetry reasons it suffices to verify that $\bs{x} := \min(\bs{x}^{(1)},\bs{x}^{(2)}) \in \mathcal{Q}$, and this is equivalent to $\bs{x}$ satisfying property~(iii) in Proposition~\ref{Prop:Characterization}. Let $1 \le a \le b \le m$, and suppose that $x_{a-1} < x_a$. Then for some $k \in \{1,2\}$,
\[
	x_{a-1} = x^{(k)}_{a-1} < x_a \le x_a^{(k)} ,
\]
so property~(iii) of $\bs{x}^{(k)}$ implies that $x_a \le x_a^{(k)} \le U_{ab}$. In case of $x_b < x_{b+1}$, we choose $k \in \{1,2\}$ such that
\[
	x_{b} = x^{(k)}_b < x_{b+1} \le x_{b+1}^{(k)} ,
\]
and then property~(iii) of $\bs{x}^{(k)}$ implies that $x_b = x_b^{(k)} \ge L_{ab}$.
\end{proof}

Now we provide the main result involving min-max and max-min formulae for the set $\mathcal{Q}$.

\begin{Theorem}
\label{Thm:Monotone_Regression}
For any index $1 \le j \le m$,
\begin{align*}
	\ell_j^{(1)} := \max_{a \le j} \, \min_{b \ge j} \, L_{ab} \
		&= \ \ell_j^{(2)} := \min_{b \ge j} \, \max_{a \le j} \, L_{ab}
\intertext{and}
	u_j^{(1)} := \min_{b \ge j} \, \max_{a \le j} \, U_{ab} \
		&= \ u_j^{(2)} := \max_{a \le j} \, \min_{b \ge j} \, U_{ab} .
\end{align*}
This defines vectors $\bs{\ell} = (\ell_j^{(1)})_{j=1}^m$ and $\bs{u} = (u_j^{(1)})_{j=1}^m$ in $\mathcal{Q}$, and any vector $\bs{x} \in \mathcal{Q}$ satisfies $\bs{\ell} \le \bs{x} \le \bs{u}$ componentwise.
\end{Theorem}

\begin{proof}[\textbf{Proof of Theorem~\ref{Thm:Monotone_Regression}}]
For symmetry reasons, if suffices to verify the claims about $\bs{\ell}$. Precisely, with $\bs{\ell}^{(k)} := (\ell_k^{(k)})_{j=1}^m$, we show subsequently that
\begin{align}
\label{ineq:MR1}
	& \bs{\ell}^{(1)} \ \le \ \bs{\ell}^{(2)} , \\
\label{ineq:MR2}
	& \bs{\ell}^{(2)} \ \le \ \bs{x} \ \quad\text{for any} \ \bs{x} \in \mathcal{Q} , \\
\label{eq:MR3}
	& \bs{\ell}^{(1)} \ \in \ \mathcal{Q} .
\end{align}

Inequality~\eqref{ineq:MR1} follows from
\[
	\ell_j^{(1)} \ \le \ \max_{a \le j} \, \min_{b \ge j} \,
		\max_{\tilde{a} \le j} \, L_{\tilde{a}b}
	\ = \ \max_{a \le j} \, \ell_j^{(2)} \ = \ \ell_j^{(2)}
\]
for $1 \le j \le m$.

As to \eqref{ineq:MR2}, for $\bs{x} \in \mathcal{Q}$ and $1 \le j \le m$ let $\tilde{b}$ be the largest index such that $x_{\tilde{b}} = x_j$. Then $x_{\tilde{b}} < x_{\tilde{b}+1}$, so property~(ii) of $\bs{x}$ in Proposition~\ref{Prop:Characterization} implies that
\[
	\ell_j^{(2)} \ \le \ \max_{a \le j} \, L_{a,\tilde{b}} \ \le \ x_{\tilde{b}} \ = \ x_j .
\]

It remains to verify~\eqref{eq:MR3}. For indices $1 \le j < k \le m$,
\[
	\ell_j^{(1)} \ = \ \max_{a \le j} \, \min_{b \ge j} \, L_{ab}
	\ \le \ \max_{a \le j} \, \min_{b \ge k} \, L_{ab}
	\ \le \ \max_{a \le k} \, \min_{b \ge k} \, L_{ab} \ = \ \ell_k^{(1)} ,
\]
whence $\bs{\ell}^{(1)} \in \Rmua$. To show that $\bs{\ell}^{(1)} \in \mathcal{Q}$, it suffices to show that it has property~(iii) in Proposition~\ref{Prop:Characterization}, and this is an immediate consequence of the following two claims: For $1 \le j \le m$,
\begin{align}
\label{ineq:MR3a}
	\ell_{j-1}^{(1)} \ < \ \ell_j^{(1)} \quad&\text{implies that}\quad
		\ell_j^{(1)} \ = \ \min_{b \ge j} \, L_{jb} , \\
\label{ineq:MR3b}
	\ell_j^{(1)} \ < \ \ell_{j+1}^{(1)} \quad&\text{implies that}\quad
		\ell_j^{(1)} \ = \ \max_{a \le j} \, L_{aj} .
\end{align}

As to \eqref{ineq:MR3a}, suppose that the conclusion is wrong, i.e.\ $\ell_j^{(1)} > \min_{b \ge j} L_{jb}$. Then $j > 1$, and for some index $\tilde{a} \le j-1$,
\[
	\ell_j^{(1)} \ = \ \min_{b \ge j} \, L_{\tilde{a}b}
	\ \le \ \min_{b \ge j} \, \max(L_{\tilde{a},j-1}, L_{jb})
	\ = \ \max \Bigl( L_{\tilde{a},j-1}, \min_{b \ge j} \, L_{jb} \Bigr)
	\ = \ L_{\tilde{a},j-1} ,
\]
where we used Proposition~\ref{Prop:Cauchy.mean.value}. But then
\[
	\ell_{j-1}^{(1)} \ \ge \ \min_{b \ge j-1} \, L_{\tilde{a}b}
	\ = \ \min \Bigl( L_{\tilde{a},j-1}, \min_{b \ge j} \, L_{\tilde{a}b} \Bigr)
	\ = \ \ell_j^{(1)} ,
\]
i.e.\ the assumption of \eqref{ineq:MR3a} is wrong as well.

Concerning \eqref{ineq:MR3b}, suppose that that the conclusion is wrong, i.e.\ $\ell_j^{(1)} < L_{\tilde{a}j}$ for some $\tilde{a} \le j$. Then $j < m$, and
\begin{align*}
	\ell_j^{(1)} \
	&\ge \ \min_{b \ge j} \, L_{\tilde{a}b}
		\ = \ \min \Bigl( L_{\tilde{a}j}, \min_{b \ge j+1} \, L_{\tilde{a}b} \Bigr)
		\ = \ \min_{b \ge j+1} \, L_{\tilde{a}b} \\
	&\ge \ \min_{b \ge j+1} \, \min(L_{\tilde{a}j}, L_{j+1,b})
		\ = \ \min \Bigl( L_{\tilde{a}j}, \min_{b \ge j+1} \, L_{j+1,b} \Bigr)
		\ = \ \min_{b \ge j+1} \, L_{j+1,b} .
\end{align*}
Consequently,
\[
	\min_{b\ge j+1} \, L_{j+1,b} \ \le \ \ell_j^{(1)}
	\quad\text{and}\quad
	\min_{b\ge j+1} \, L_{\tilde{a}b} \ \le \ \ell_j^{(1)} .
\]
This is true for any index $\tilde{a} \le j$ with $L_{\tilde{a}j} > \ell_j^{(1)}$. If $a \le j$ is such that $L_{aj} \le \ell_j^{(1)}$, then
\[
	\min_{b \ge j+1} \, L_{ab}
	\ \le \ \min_{b \ge j+1} \, \max(L_{aj}, L_{j+1,b})
	\ = \ \max \Bigl( L_{aj}, \min_{b \ge j+1} \, L_{j+1,b} \Bigr)
	\ \le \ \ell_j^{(1)} .
\]
Thus $\min_{b \ge j+1} L_{aj} \le \ell_j^{(1)}$ for any $a \le j+1$. Consequently, $\ell_{j+1}^{(1)} \le \ell_j^{(1)}$, i.e.\ the assumption of \eqref{ineq:MR3b} is wrong as well.
\end{proof}

We end this subsection with two additional conclusions for the special case of convex functions $R_j$.

\begin{Theorem}
\label{Thm:Monotone_Regression_convex}
Suppose in addition that all loss functions $R_j$ are convex. Then the set $\mathcal{Q}$ is compact and convex. If $\bs{x} \in \Rmua$ is such that $\bs{\ell} \le \bs{x} \le \bs{u}$ and $\{j < m : x_j < x_{j+1}\} \subset \{j < m : \ell_j < \ell_{j+1} \ \text{or} \ u_j < u_{j+1}\}$, then $\bs{x} \in \mathcal{Q}$. Moreover, each function $R_j$ is linear on the interval $[\ell_j,u_j]$.
\end{Theorem}

\begin{proof}[\textbf{Proof}]
The general assumptions imply that each function $R_j = R_{jj}$ has a compact set of minimizers. Together with convexity, this implies that $R_j$ is continuous with $R_j(x) \to \infty$ as $|x| \to \infty$. But then, $T : \R^m \to \R$ is a continuous and convex function such that $T(\bs{x}) \to \infty$ as $\|\bs{x}\| \to \infty$. Moreover, $\Rmua$ is a closed convex cone in $\R^m$. This implies that $\mathcal{Q}$ is a compact and convex set.

To verify the remaining statements, consider the vectors $\bs{x}(\lambda) := (1 - \lambda) \bs{\ell} + \lambda \bs{u}$, $\lambda \in [0,1]$. Since $\mathcal{Q}$ is a convex set, all these vectors belong to $\mathcal{Q}$. But for $0 < \lambda < 1$,
\[
	\{j < m : x_j(\lambda) < x_{j+1}(\lambda)\}
	\ = \ \{j < m : \ell_j < \ell_{j+1} \ \text{or} \ u_j < u_{j+1}\} .
\]
Exploiting property~(ii) of $\bs{x}(\lambda)$ in Proposition~\ref{Prop:Characterization} for all $\lambda \in (0,1)$, we may conclude that for arbitrary indices $1 \le a \le b \le m$,
\begin{align*}
	&u_a \ \le \ U_{ab} \quad\text{if} \ \ell_{a-1} < \ell_a \ \text{or} \ u_{a-1} < u_a , \\
	&\ell_b \ \ge \ L_{ab} \quad\text{if} \ \ell_b < \ell_{b+1} \ \text{or} \ u_b < u_{b+1} .
\end{align*}
In particular, any vector $\bs{x} \in \Rmua$ such that $\bs{\ell} \le \bs{x} \le \bs{u}$ and $\{j < m : x_j < x_{j+1}\}$ is a subset of $\{j < m : \ell_j < \ell_{j+1} \ \text{or} \ u_j < u_{j+1}\}$ satisfies property~(iii) in Proposition~\ref{Prop:Characterization}. Hence $\bs{x} \in \mathcal{Q}$.

Finally, since
\[
	T_\beta(\bs{q}(\lambda)) \ = \ \sum_{j=1}^m R_j\bigl((1-\lambda)\ell_j + \lambda u_j\bigr) 
\]
is constant in $\lambda\in[0,1]$, each summand $R_j\bigl((1-\lambda)\ell_j + \lambda u_j\bigr)$ has to be linear in $\lambda \in [0,1]$, which is equivalent to $R_j$ being linear on $[\ell_j,u_j]$.
\end{proof}

\subsection{Proofs of Lemma~\ref{lem:Qhat.RQ} and \ref{lem:Qhat.LS}}

\begin{proof}[\textbf{Proof of Lemma~\ref{lem:Qhat.RQ}}]
For $1\leq j\leq m$, set
\[
	R_j(q) \ := \ \sum_{i:\,X_i=x_j} \rho_{\beta}(Y_i - q) .
\]
This is a convex function of $q \in \R$ with $R_j(q) \to \infty$ as $|q| \to \infty$. To apply the results of the previous subsection, we need to determine the sets $[L_{ab},U_{ab}]$ for $1 \le a \le b \le m$. Note that $R_j'(q\,+) = \sum_{i : X_i = x_j} (1_{[Y_i \le q]} - \beta)$, whence
\[
	R_{ab}'(q\,+) \ = \ w_{ab} (\Fh_{ab}(q) - \beta) .
\]
Consequently,
\begin{align*}
	L_{ab} \ &= \ \min \bigl\{ q \in \R : R_{ab}'(q\,+) \ge 0 \bigr\}
		\ = \ \Fh_{ab}^{-1}(\beta) , \\
	U_{ab} \ &= \ \inf \bigl\{ q \in \R : R_{ab}'(q\,+) > 0 \bigr\}
		\ = \ \Fh_{ab}^{-1}(\beta\,+) .
\end{align*}
Now all but the last statement of Lemma~\ref{lem:Qhat.RQ} follow from Theorems~\ref{Thm:Monotone_Regression} and \ref{Thm:Monotone_Regression_convex}. As to the last statement, note that each $R_j$ is a convex and piecewise linear function with strict changes of slope at each $Y_i$ such that $X_i=x_j$. Consequently, since $R_j$ is linear on $[\ell_j,u_j]$, there is no data point $(X_i,Y_i)$ such that $X_i = x_j$ and $Y_i \in (\ell_j,u_j)$.
\end{proof}

\begin{proof}[\textbf{Proof of Lemma~\ref{lem:Qhat.LS}}]
For arbitrary $y \in \R$,
\[
	y \ \ge \ \hat{F}_{x_j}^{-1}(\beta)
	\quad\text{if and only if}\quad
	\hat{F}_{x_j}(y) \ \ge \ \beta .
\]
But the min-max formula \eqref{eq:minmaxmin.LS} for $\hat{F}_{x_j}(y)$ implies that the inequality on the right hand side is equivalent to the following statements:
\begin{align*}
	&\min_{r \le j} \, \max_{s \ge j} \, \Fh_{rs}(y) \ \ge \ \beta , \\
	&\text{for all} \ r \le j, \quad
		\Fh_{rs}(y) \ \ge \ \beta
		\ \ \text{for some} \ s = s(r) \ge j , \\
	&\text{for all} \ r \le j, \quad
		y \ \ge \ \Fh_{rs}^{-1}(\beta)
		\ \ \text{for some} \ s = s(r) \ge j , \\
	&y \ \ge \ \max_{r \le j} \, \min_{s \ge j} \, \Fh_{rs}^{-1}(\beta)
		\ = \ \ell_j .
\end{align*}
Hence $\hat{F}_{x_j}^{-1}(\beta) = \ell_j$. Analogously, for any $y \in \R$,
\[
	y \ \ge \ \hat{F}_{x_j}^{-1}(\beta\,+)
	\quad\text{if and only if}\quad
	\hat{F}_{x_j}(y\,-) \ \le \ \beta .
\]
But \eqref{eq:minmaxmin.LS} remains valid if we replace `$(y)$' with `$(y\,-)$', so the inequality on the right hand side is equivalent to the following statements:
\begin{align*}
	&\max_{s \ge j} \, \min_{r \le j} \, \Fh_{rs}(y\,-) \ \le \ \beta , \\
	&\text{for all} \ s \ge j, \quad
		\Fh_{rs}(y\,-) \ \le \ \beta
		\ \ \text{for some} \ r = r(s) \le j , \\
	&\text{for all} \ s \ge j, \quad
		y \ \le \ \Fh_{rs}^{-1}(\beta\,+)
		\ \ \text{for some} \ r = r(s) \ge j , \\
	&y \ \le \ \min_{s \ge j} \, \max_{r \le j} \, \Fh_{rs}^{-1}(\beta\,+)
		\ = \ u_j .
\end{align*}
Hence $\hat{F}_{x_j}^{-1}(\beta\,+) = u_j$.
\end{proof}

\subsection{Asymptotics}

In what follows, we always work with the conditional distribution of $(Y_{ni})_{i=1}^n$, given $\Xb_n$. Moreover, we tacitly assume that $\Xb_n$ is a ``good'' vector in the sense that the event $A_n$ in Assumption~(A.2) or (A'.2$_{x_o}$) occurs.

To lighten the notation, we do not introduce an extra subscript $n$ for the weights $w_{rs}$ or the empirical distribution functions $\Fh_{rs}$. Furthermore, we define
\[
	\bar{F}_{rs}(\fundot) \ := \ w_{rs}^{-1} \sum_{j=r}^s w_j F_{x_j}(\fundot) .
\]
The norm $\|\cdot\|_{\infty}$ denotes the usual supremum norm of functions on the real line.

The proofs make use of the following exponential inequality which follows from \citet{Bretagnolle1980} and \citet{Hu_1985}.

\begin{Theorem}
\label{Thm:Bretagnolle}
Let $Y_1, Y_2, Y_3, \ldots$ be independent random variables with respective distribution functions $F_1,F_2,F_3, \ldots$\,. For $k \in \N$, let
\[
	\hat{\F}(\fundot) \ := \ \frac{1}{k} \sum_{i=1}^k 1_{[Y_i\leq \fundot]}
	\quad\text{and}\quad
	\bar{F}(\fundot) \ := \ \frac{1}{k} \sum_{i=1}^k F_i(\fundot) .
\]
Then there exists a universal constant $C_4 \le 2^{5/2} e$ such that for all $\eta \ge 0$,
\[
	\Pr \Bigl(
		\sqrt{k} \bigl\| \hat{\F} - \bar{F} \bigr\|_{\infty}
		\geq \eta \Bigr)
	\ \leq \ C_4 \exp(- 2\eta^2) .
\]
\end{Theorem}

\begin{Corollary}
\label{Cor:ConsisEmpirical}
Let
\[
	M_n \ := \ \max_{1 \le r \le s \le m} \,
		w_{rs}^{1/2} \|\Fh_{rs} - \bar{F}_{rs}\|_{\infty} .
\]
Then for any constant $D > 1$,
\[
	\lim_{n \to \infty} \, \Pr(M_n \le (D \log n)^{1/2}) \ = \ 1 .
\]
\end{Corollary}

\begin{proof}[\textbf{Proof of Corollary~\ref{Cor:ConsisEmpirical}}]
Note that $M_n$ is the maximum of the $\binom{m+1}{2}$ quantities
\[
	w_{rs}^{1/2} \|\Fh_{rs} - \bar{F}_{rs}\|_{\infty} ,
\]
and we may apply Theorem~\ref{Thm:Bretagnolle} to each of them. Consequently,
\begin{align*}
	\Pr(M_n \ge \eta_n) \
	&\le \ \sum_{1 \le r \le s \le m}
		\Pr \bigl( w_{rs}^{1/2} \|\Fh_{rs} - \bar{F}_{rs}\|_{\infty} \ge \eta_n \bigr) \\
	&\le \ C_4 \binom{m}{2} \exp(- 2\eta_n^2) \\
	&\le \ (C_4/2) \exp(2 \log(n+1) - 2 \eta_n^2)
\end{align*}
for arbitrary $\eta_n \ge 0$. But the right hand side converges to zero as $n \to \infty$ if $\eta_n = (D \log n)^{1/2}$ for some $D > 1$.
\end{proof}

\begin{proof}[\textbf{Proof of Theorem~\ref{Thm:F_Asymptotics}}]
Recall that $\rho_n = \log(n)/n$, $\delta_n = C_3\rho_n^{1/(2\alpha+1)}$ and $I_n = \{x \in I : [x \pm \delta_n] \subset I\}$. Recall also that we treat $\bs{X}_n$ as fixed and assume that the event $A_n$ in Assumption~(A.2) occurs. Let $n$ be sufficiently large so that $I_n \neq \emptyset$. For $x \in I_n$ the indices
\begin{align*}
	r(x) \ &:= \ \min \bigl\{ j \in \{1,\ldots,m\} : x_j \ge x - \delta_n \bigr\} , \\
	j(x) \ &:= \ \max \bigl\{ j \in \{1,\ldots,m\} : x_j \le x \bigr\}
\end{align*}
are well-defined, because $[x - \delta_n, x]$ is a subinterval of $I$ of length $\delta_n$, so Assumption~(A.2) guarantees that this interval contains at least one observation $x_j$. Moreover,
\[
	r(x) \le j(x) , \quad
	x - \delta_n \le x_{r(x)} \le x_{j(x)} \le x
	\quad\text{and}\quad
	w_{r(x)j(x)} = w_n([x-\delta_n,x]) \ge C_2n\delta_n .
\]
Consequently, with $M_n$ as in Corollary~\ref{Cor:ConsisEmpirical}, for any $y \in J$ we obtain the inequalities
\begin{align*}
	\hat{F}_{nx}(y) - F_x(y) \
	&\le \ \hat{F}_{nx_{j(x)}}(y) - F_x(y) \\
	&= \ \min_{r \le j(x)} \, \max_{s \ge j(x)} \, \Fh_{rs}(y) - F_x(y) \\
	&\le \ \max_{s \ge j(x)} \, \Fh_{r(x)s}(y) - F_x(y) \\
	&\le \ w_{r(x)j(x)}^{-1/2} M_n
		+ \max_{s \ge j(x)} \, \bar{F}_{r(x)s}(y) - F_x(y) \\
	&\le \ (C_2 n\delta_n)^{-1/2} M_n
		+ F_{x_{r(x)}}(y) - F_x(y) \\
	&\le \ (C_2 n\delta_n)^{-1/2} M_n
		+ C_1 \delta_n^\alpha .
\end{align*}
In the first step we used antitonicity of $\tilde{x} \mapsto \hat{F}_{n\tilde{x}}(y)$, in the second last step we used antitonicity of $\tilde{x} \mapsto F_{\tilde{x}}(y)$, and the last step utilizes Assumption~(A.1). But $\Pr(M_n \le (D \log n)^{1/2}) \to 1$ for any fixed $D > 1$, and on the event $\{M_n \le (D \log n)^{1/2}\}$, the previous considerations imply that
\[
	\sup_{x \in I_n, y \in J} \bigl( \hat{F}_{nx}(y) - F_x(y) \bigr)
	\ \le \ (C_2 n\delta_n)^{-1/2} (D \log n)^{1/2} + C_1 \delta_n^\alpha
	\ = \ C \rho_n^{\alpha/(2\alpha + 1)}
\]
with $C := (C_2D/C_3)^{1/2} + C_1 C_3^\alpha$.

Analogously one can show that on $\{M_n \le (D \log n)^{1/2}\}$,
\[
	\sup_{x \in I_n, y \in J} \bigl( F_x(y) - \hat{F}_{nx}(y) \bigr)
	\ \le \ (n\delta_n)^{-1/2} (D \log n)^{1/2} + C_1 \delta_n^\alpha
	\ = \ C \rho_n^{\alpha/(2\alpha + 1)}
\]
with the same constant $C$.
\end{proof}

The proof of Theorem~\ref{Thm:Q_Asymptotics} is based on Theorem~\ref{Thm:F_Asymptotics} and two elementary inequalities for distribution functions:

\begin{Lemma}
\label{Lem:Quant_Growth}
Suppose that $F,G$ are distribution functions such that
\[
	\|F - G\|_\infty \ \leq \ \Delta \ < \ 1.
\]
Then
\begin{align*}
	G^{-1}(\beta)  \ & \geq \ F^{-1}(\beta - \Delta), \qquad \text{for}\ \Delta < \beta < 1,\\
	G^{-1}(\beta+) \ & \geq \ F^{-1}((\beta + \Delta)+), \qquad \text{for}\ 0 < \beta < 1 -\Delta.
\end{align*}
\end{Lemma}

\begin{Lemma}
\label{Lem:Quant_Continuity}
Suppose that $F$ is a distribution function so that, for given $0\leq \beta_1 < \beta_2 \leq 1$ and $\kappa>0$,
\[
	F(y_2) - F(y_1) \ \geq \ \kappa(y_2 - y_1)
\]
for arbitrary $y_1 < y_2$ such that $F(y_1),F(y_2-)\in (\beta_1,\beta_2)$. Then $F^{-1}(\beta)=F^{-1}(\beta+)$ and
\begin{equation}
\label{Eq:Quant_Continuity}
	\bigl| F^{-1}(\beta) - F^{-1}(\beta') \bigr|
	\ \leq \
	\kappa^{-1} | \beta - \beta' |,
\end{equation}
for arbitrary $\beta,\beta'\in(\beta_1,\beta_2)$.
\end{Lemma}

\begin{proof}[\textbf{Proof of Lemma~\ref{Lem:Quant_Growth}}]
Let $\Delta < \beta < 1$ and $y < F^{-1}(\beta - \Delta)$. Then $F(y) < \beta - \Delta$ and thus
\[
	G(y) 
	\ \leq \ F(y) + \Delta 
	\ <    \ \beta - \Delta + \Delta
	\ =    \ \beta.
\]
Therefore, we have $y < G^{-1}(\beta)$ and letting $y\to F^{-1}(\beta - \Delta)$ yields the first inequality.

Next, let $0 < \beta < 1 -\Delta$ and $y > F^{-1}((\beta + \Delta)+)$. Then $F(y-)>\beta+\Delta$ and thus
\[
	G(y-)
	\ \geq \ F(y-) - \Delta
	\ >    \ \beta + \Delta - \Delta
	\ =    \ \beta.
\]
This gives $y > G^{-1}(\beta +)$, and letting $y\to F^{-1}((\beta - \Delta)+)$ proves the second claim.
\end{proof}

\begin{proof}[\textbf{Proof of Lemma~\ref{Lem:Quant_Continuity}}]
Let $\beta,\beta'\in(\beta_1,\beta_2)$ be such that $\beta < \beta'$. Define $y_1:=F^{-1}(\beta)$ and $y_2:=F^{-1}(\beta')$, so that $y_1\leq y_2$. If $y_1=y_2$, then \eqref{Eq:Quant_Continuity} is trivial. In case $y_1<y_2$, we have, for all $h \in (0,y_2-y_1]$, that
\[
	\beta_1
	\ <    \ \beta
	\ \leq \ F(y_1)
	\ \leq \ F(y_2 - h)
	\ \leq \ \beta'
	\ <    \ \beta_2,
\]
so that $F(y_1),F(y_2-h)\in(\beta_1,\beta_2)$. Therefore, we get
\[
	\beta' - \beta
	\ \geq \ \lim_{h\downarrow 0} F(y_2-h) - F(y_1)
	\ \geq \ \lim_{h\downarrow 0} \kappa(y_2 - h - y_1)
	\ =    \ \kappa(F^{-1}(\beta') - F^{-1}(\beta)).
\]\\[-5ex]
\end{proof}

\begin{proof}[\textbf{Proof of Theorem~\ref{Thm:Q_Asymptotics}}]
With $\Delta_n := C \rho_n^{\alpha/(2\alpha+1)}$, we may write $B_n= (\beta_1 + \Delta_n, \beta_2 - \Delta_n)$. Let $n$ be large enough so that $I_n$ and $B_n$ are nondegenerate intervals; in particular, $\Delta_n < 1/2$. The proof of Theorem~\ref{Thm:F_Asymptotics} reveals that $\Pr(A_n^*) \to 1$, where $A_n^*$ is the event that
\[
	\sup_{x \in I_n} \, \|\hat{F}_{nx,k} - F_x\|_\infty \ \le \ \Delta_n
	\quad\text{for} \ k = 1,2 .
\]
Here $\hat{F}_{nx,1}$ and $\hat{F}_{nx,2}$ denote two extremal ways to extrapolate $\hat{F}_{nx}$ from $x \in \{x_1,\ldots,x_m\}$ to arbitrary $x \in \XX$: With $x_0 := -\infty$ and $x_{m+1} := \infty$, we define
\begin{align*}
	\hat{F}_{nx,1} \ := \ \begin{cases}
		\hat{F}_{nx_j} & \text{if} \ x_{j-1} < x \le x_j, \ 1 \le j \le m , \\
		0 & \text{if} \ x > x_m ,
	\end{cases} \\
	\hat{F}_{nx,2} \ := \ \begin{cases}
		1 & \text{if} \ x < x_1 , \\
		\hat{F}_{nx_j} & \text{if} \ x_j \le x < x_{j+1}, \ 1 \le j \le m .
	\end{cases}
\end{align*}
Then $\hat{F}_{nx,1} \ge \hat{F}_{nx} \ge \hat{F}_{nx,2}$ for any choice of $(\hat{F}_x)_{x \in \XX}$. The event $A_n^*$ implies that $\hat{F}_{nx,k}$ is a proper distribution function for $k = 1,2$ and all $x \in I_n$. Moreover, for $x \in I_n$ and $\beta \in B_n$, it follows from Lemmas~\ref{Lem:Quant_Growth} and \ref{Lem:Quant_Continuity} that
\begin{align*}
	\hat{Q}_x(\beta) \ &
	\ge \ \hat{F}_{nx,1}^{-1}(\beta)
	\ \ge \ F_x^{-1}(\beta - \Delta_n)
	\ \ge \ F_x^{-1}(\beta) - \kappa^{-1} \Delta_n , \\
	\hat{Q}_x(\beta) \ &
	\le \ \hat{F}_{nx,2}^{-1}(\beta\,+)
	\ \le \ F_x^{-1}((\beta + \Delta_n)\,+)
	\ \le \ F_x^{-1}(\beta) + \kappa^{-1} \Delta_n .
\end{align*}
Consequently,
\[
	\Pr \biggl( \sup_{x\in I_n, \beta\in B_n} \, 
			\bigl| \hat{Q}_x(\beta) - Q_x(\beta) \bigr|
		> \kappa^{-1} \Delta_n \biggr)
	\ \ge \ \Pr(A_n^*) \ \to \ 1
\]
as $n\to \infty$.
\end{proof}

We now proceed to the proof of Theorem~\ref{Thm:F_Asymptotics_xo}. Theorem~\ref{Thm:Bretagnolle}
and Lemma~\ref{Lem:ExpIneq} in the next subsection imply the following exponential inequality:

\begin{Corollary}
\label{Cor:Bretagnolle}
With the same notation as in Theorem~\ref{Thm:Bretagnolle}, for any $D' \in (0,2)$ there exists a universal constant $D'' = D''(D')$ such that
\[
	\Pr \Bigl(
		\sup_{k\geq k_o} \,
		\bigl\| \hat{\F}_k - \bar{F}_k \bigr\|_{\infty} \geq \eta	
	\Bigr) \ \leq \ D'' \exp(- D' k_o \eta^2)
\]
for all $k_o\in\N$ and $\eta\geq 0$.
\end{Corollary}

\begin{proof}[\textbf{Proof of Theorem~\ref{Thm:F_Asymptotics_xo}}]
Let us define the indices
\[
	r_n \ := \ \min \bigl\{ j\in\{1,\ldots,m\} : x_j \ge x_o - \delta_n \bigr\}
	\quad\text{and}\quad
	j_n \ := \ \max \bigl\{ j\in\{1,\ldots,m\} : x_j \le x_o \bigr\} .
\]
Since we assume the event $A_n$ in (A'.2$_{x_o}$) to occur, we know that
\[
	x_o - \delta_n \le x_{r_n} \le x_{j_n} \le x_o
	\quad
	\text{and}\quad
	w_{r_nj_n} = \ w_n([x_o-\delta_n,x_o]) \ \geq \
	C_2 n \delta_n \ > \ 0 .
\]
One can easily deduce from Corollary~\ref{Cor:Bretagnolle} that
\[
	M_n \ := \ \max_{j \ge j_n} \, w_{r_nj}^{1/2} \|\Fh_{r_nj} - \bar{F}_{r_nj}\|_\infty
	\ = \ O_p(1) .
\]
Consequently, for $y \in J$,
\begin{align*}
	\hat{F}_{nx_o}(y) - F_{x_o}(y) \
	&\le \ \hat{F}_{nj_n}(y) - F_{x_o}(y) \\
	&= \ \min_{r \le j_n} \, \max_{s \ge j_n} \, \Fh_{rs}(y) - F_{x_o}(y) \\
	&\le \ \max_{s \ge j_n} \, \Fh_{r_ns}(y) - F_{x_o}(y) \\
	&\le \ w_{r_nj_n}^{-1/2} M_n + \max_{s \ge j_n} \, \bar{F}_{r_n,s}(y) - F_{x_o}(y) \\
	&\le \ (C_2 n \delta_n)^{-1/2} M_n + F_{x_o-\delta_n}(y) - F_{x_o}(y) \\
	&\le \ (C_2 n \delta_n)^{-1/2} M_n + C_1 \delta_n^\alpha .
\end{align*}
But the right hand side does not depend on $y$ and is of order $O_p \bigl( (n \delta_n)^{-1/2} + \delta_n^\alpha \bigr) = O_p(n^{-\alpha/(2\alpha + 1)})$. Consequently,
\[
	\sup_{y \in J} \bigl( \hat{F}_{x_o}(y) - F_{x_o}(y) \bigr)
	\ = \ O_p(n^{-\alpha/(2\alpha + 1)}) .
\]
Analogous arguments show that $\sup_{y \in J} \bigl( F_{x_o}(y) - \hat{F}_{x_o}(y) \bigr)$ is of order $O_p(n^{-\alpha/(2\alpha + 1)})$, too.
\end{proof}

\begin{proof}[\textbf{Proof of Theorem~\ref{Thm:Q_Asymptotics_xo}}]
The proof uses essentially the same arguments as the proof of Theorem~\ref{Thm:Q_Asymptotics}. The main differences are that we replace $I_n$ with $\{x_o\}$ and $\rho_n$ with $n^{-1}$.
\end{proof}

\subsection{An exponential inequality for the LLN}

We consider stochastically independent random elements $Z_1, Z_2, Z_3, \ldots$ with values in a normed vector space $(\mathcal{Z},\|\cdot\|)$. Defining the partial sums $S_0 := 0$ and $S_n := \sum_{i=1}^n Z_i$ for $n \in \mathbb{N}$, we assume that $\|S_b - S_a\|$ is measurable for arbitrary integers $0 \le a < b$.

\begin{Lemma}
\label{Lem:ExpIneq}
Suppose that there are constants $c > 0$ and $C \ge 1$ such that for arbitrary integers $0 \le a < b$ and real numbers $\eta > 0$,
\begin{equation}
\label{Exp.ineq.1}
	\Pr(\|S_b - S_a\| > \eta) \ \le \ C \exp \bigl( - c \eta^2/(b - a) \bigr) .
\end{equation}
Then for arbitrary $c' \in (0,c)$ there exists a constant $C'$ such that
\begin{equation}
\label{Exp.ineq.2}
	\Pr \Bigl( \sup_{n \ge n_o} \|S_n/n\| \ge \eta \Bigr)
	\ \le \ C' \exp( - c' n_o \eta^2)
\end{equation}
for arbitrary numbers $n_o, \eta \ge 0$.
\end{Lemma}

Corollary~\ref{Cor:Bretagnolle} is a consequence of this result, where $Z_i := 1_{[Y_i \le \fundot]} - F_i$ is a random bounded function on the real line, and $c = 2$.

\begin{proof}[\bf Proof of Lemma~\ref{Lem:ExpIneq}]
Note that the right hand side of \eqref{Exp.ineq.2} is continuous in $\eta \ge 0$ and $n_o \ge 0$, and it is not smaller than $1$ in case of $\eta = 0$ or $n_o = 0$. Hence it suffices to verify that
\begin{equation}
\label{Exp.ineq.2'}
	\Pr \Bigl( \sup_{n \ge n_o} \|S_n/n\| > \eta \Bigr)
	\ \le \ C' \exp( - c' n_o \eta^2)
\end{equation}
for arbitrary numbers $n_o, \eta > 0$.

The essential ingredient will be the following inequality: For arbitrary real numbers $0 \le a < b$ and $\eta > 0$,
\begin{equation}
\label{Exp.ineq.3}
	\Pr \Bigl( \max_{a \le n \le b} \|S_n\| > \eta \bigr)
	\ \le \ 2 C \exp \biggl( - \frac{c \eta^2}{\bigl( \sqrt{b} + \sqrt{b-a} \bigr)^2} \biggr)
\end{equation}
(with the maximum over the empty set interpreted as $0$). To verify this, it suffices to consider the case of $a$ and $b$ being integers; otherwise one could replace $a$ with $\lceil a\rceil$ and $b$ with $\lfloor b\rfloor$, and this would even decrease the term $\sqrt{b} + \sqrt{b-a}$ in \eqref{Exp.ineq.3}. Define the stopping time
\[
	\tau
	\ := \ \min \bigl( \bigl\{ n \in \{a,\ldots,b\} : \|S_n\| > \eta \bigr\}
		\cup \{\infty\} \bigr) .
\]
Then, for $0 < \lambda < 1$,
\begin{align*}
	\Pr \Bigl( \max_{a \le n \le b} \|S_n\| > \eta \bigr) \
	&= \ \Pr(\tau \le b) \\
	&\le \ \Pr(\|S_b\| > \lambda \eta)
		+ \Pr \bigl( \tau \le b, \|S_b\| \le \lambda \eta \bigr) \\
	&= \ \Pr(\|S_b\| > \lambda \eta) + \sum_{n=a}^{b-1}
		\Pr \bigl( \tau = n, \|S_b\| \le \lambda \eta \bigr) \\
	&\le \ \Pr(\|S_b\| > \lambda \eta) + \sum_{n=a}^{b-1}
		\Pr \bigl( \tau = n, \|S_n - S_b\| > (1 - \lambda) \eta \bigr) \\
	&= \ \Pr(\|S_b\| > \lambda \eta) + \sum_{n=a}^{b-1}
		\Pr(\tau = n) \Pr \bigl( \|S_n - S_b\| > (1 - \lambda) \eta \bigr) \\
	&\le \ C \exp \biggl( - \frac{c \lambda^2 \eta^2}{b} \biggr)
		+ \sum_{n=a}^{b-1} \Pr(\tau = n) \,
			C \exp \biggl( - \frac{c (1 - \lambda)^2 \eta^2}{b-a} \biggr) \\
	&\le \ C \exp \biggl( - \frac{c \lambda^2 \eta^2}{b} \biggr)
		+ C \exp \biggl( - \frac{c (1 - \lambda)^2 \eta^2}{b-a} \biggr) .
\end{align*}
Here the fourth last step follows from the triangle inequality for $\|\cdot\|$: $\|S_n - S_b\| \ge \|S_n\| - \|S_b\| > \eta - \lambda \eta$ in case of $\tau = n$ and $\|S_b\| \le \lambda \eta$. The third last step follows from independence of the $Z_i$ and the fact that the event $\{\tau = n\}$ depends on $Z_a,\ldots,Z_n$, whereas $\|S_n - S_b\|$ is a function of $Z_{n+1},\ldots,Z_b$. If we take
\[
	\lambda \ := \ \frac{\sqrt{b}}{\sqrt{b} + \sqrt{b - a}} ,
\]
then the two exponents in our inequality are identical, and we obtain \eqref{Exp.ineq.3}.

Since $c' < c$, the constant
\[
	\beta \ := \ \frac{(c/c' + 1)^2}{4 c/c'}
\]
satisfies $\beta > 1$ and
\[
	c' \ = \ \frac{c}{\bigl(\sqrt{\beta} + \sqrt{\beta - 1}\bigr)^2} .
\]
With \eqref{Exp.ineq.3} at hand, we may argue that for arbitrary numbers $n_o > 0$,
\begin{align*}
	\Pr \Bigl( \sup_{n \ge n_o} \|S_n/n\| > \eta \Bigr) \
	&\le \ \sum_{k=0}^\infty
		\Pr \Bigl( \max_{\beta^k n_o \le n \le \beta^{k+1}n_o} \|S_n\| > \beta^k n_o \eta \Bigr) \\
	&\le \ 2 C \sum_{k=0}^\infty
		\exp \biggl( - \frac{c \beta^{2k} n_o^2 \eta^2}
			{\bigl( \sqrt{\beta^{k+1} n_o} + \sqrt{\beta^{k+1} n_o - \beta^{k} n_o} \bigr)^2} \biggr) \\
	&= \ 2 C \sum_{k=0}^\infty
		\exp \biggl( - \frac{c \beta^k n_o \eta^2}
			{\bigl( \sqrt{\beta} + \sqrt{\beta - 1} \bigr)^2} \biggr) \\
	&= \ 2 C \sum_{k=0}^\infty
		\exp(- p(\eta) \beta^k) ,
\end{align*}
where $p(\eta) := c' n_o \eta^2 > 0$. Since $\beta^x$ is increasing in $x \ge 0$, we find the upper bound
\begin{align*}
	\sum_{k=1}^\infty \exp(-p(\eta)\beta^k) \
	& \leq \ \int_{0}^\infty \exp(-p(\eta)\beta^x) \, dx \\
	& = \ (\log\beta)^{-1} \int_{0}^\infty \exp(-p(\eta) e^y) \, dy \\
	& \leq \ (\log\beta)^{-1} \int_{0}^\infty \exp(-p(\eta) (1 + y)) \, dy \\
	& = \ \frac{1}{p(\eta)\log\beta} \exp(-p(\eta)) ,
\end{align*}
which yields
\[
	\Pr \Bigl( \sup_{n \ge n_o} \|S_n/n\| > \eta \Bigr)
	\ \leq \
	2 C \Bigl(1 + \frac{1}{p(\eta) \log\beta} \Bigr) \exp(-p(\eta)) .
\]
For a number $p_o > 0$ to be specified later, the bound above is not greater than
\[
	2 C \Bigl(1 + \frac{1}{p_o \log\beta} \Bigr) \exp(-p(\eta))
	\ = \
	2 C \Bigl(1 + \frac{1}{p_o \log\beta} \Bigr) \exp(-c' n_o \eta^2)
\]
whenever $p(\eta) \ge p_o$. But in case of $p(\eta) \le p_o$, the latter bound is at least
\[
	2 C \Bigl(1 + \frac{1}{p_o \log\beta} \Bigr) \exp(- p_o) \ \ge \ 1
\]
if we set $p_o := \min\{(\log\beta)^{-1},\log(4C)\}$. Consequently, with this choice of $p_o$, \eqref{Exp.ineq.2'} is true with $C' := 2 C \bigl(1 + (p_o \log\beta)^{-1} \bigr)$.
\end{proof}

\paragraph{Acknowledgements.}
This work was supported by Swiss National Science Foundation. The authors are grateful to Geurt Jongbloed for drawing their attention to \cite{ElBarmi2005} and to Johanna Ziegel for stimulating discussions.



\begin{thebibliography}{13}
\expandafter\ifx\csname natexlab\endcsname\relax\def\natexlab#1{#1}\fi
\expandafter\ifx\csname url\endcsname\relax
  \def\url#1{\texttt{#1}}\fi
\expandafter\ifx\csname urlprefix\endcsname\relax\def\urlprefix{URL }\fi

\bibitem[{Bretagnolle(1980)}]{Bretagnolle1980}
\textsc{Bretagnolle, J.} (1980).
\newblock {Statistique de Kolmogorov-Smirnov pour un \'echantillon
  non\'equir\'eparti}.
\newblock \textit{Colloques Internationaux du CNRS} \textbf{307} 39--44.

\bibitem[{Casady and Cryer(1976)}]{Casady_Cryer_1976}
\textsc{Casady, R.~J.} and \textsc{Cryer, J.~D.} (1976).
\newblock Monotone percentile regression.
\newblock \textit{Ann. Statist.} \textbf{4} 532--541.

\bibitem[{D\"umbgen and Kovac(2009)}]{Duembgen2009}
\textsc{D\"umbgen, L.} and \textsc{Kovac, A.} (2009).
\newblock Extensions of smoothing via taut strings.
\newblock \textit{Electron. J. Statist.} \textbf{3} 41--75.

\bibitem[{El~Barmi and Mukerjee(2005)}]{ElBarmi2005}
\textsc{El~Barmi, H.} and \textsc{Mukerjee, H.} (2005).
\newblock Inferences under a stochastic ordering constraint.
\newblock \textit{J. Amer. Statist. Assoc.} \textbf{100} 252--261.

\bibitem[{Henzi(2018)}]{Henzi2018}
\textsc{Henzi, A.} (2018).
\newblock \textit{Isotonic Distributional Regression (IDR): A powerful
  nonparametric calibration technique}.
\newblock Master's thesis, University of Bern.

\bibitem[{Hu(1985)}]{Hu_1985}
\textsc{Hu, I.} (1985).
\newblock A uniform bound for the tail probability of {K}olmogorov-{S}mirnov
  statistics.
\newblock \textit{Ann. Statist.} \textbf{13} 821--826.

\bibitem[{Koenker and Bassett(1978)}]{Koenker1978}
\textsc{Koenker, R.} and \textsc{Bassett, G.} (1978).
\newblock Regression quantiles.
\newblock \textit{Econometrica} \textbf{46} 33--50.

\bibitem[{M\"uhlemann et~al.(2019)M\"uhlemann, Jordan and
  Ziegel}]{Muehlemann2019}
\textsc{M\"uhlemann, A.}, \textsc{Jordan, A.~I.} and \textsc{Ziegel, J.~F.}
  (2019).
\newblock Optimal solutions to the isotonic regression problem.
\newblock Preprint, arXiv:1904.04761 [math.ST].

\bibitem[{Mukerjee(1993)}]{Mukerjee_1993}
\textsc{Mukerjee, H.} (1993).
\newblock An improved monotone conditional quantile estimator.
\newblock \textit{Ann. Statist.} \textbf{21} 924--942.

\bibitem[{Poiraud-Casanova and Thomas-Agnan(2000)}]{PoiraudCasanova2000}
\textsc{Poiraud-Casanova, S.} and \textsc{Thomas-Agnan, C.} (2000).
\newblock About monotone regression quantiles.
\newblock \textit{Statist. Probab. Lett.} \textbf{48} 101--104.

\bibitem[{Robertson and Wright(1980)}]{Robertson_Wright_1980}
\textsc{Robertson, T.} and \textsc{Wright, F.~T.} (1980).
\newblock Algorithms in order restricted statistical inference and the {C}auchy
  mean value property.
\newblock \textit{Ann. Statist.} \textbf{8} 645--651.

\bibitem[{Robertson et~al.(1988)Robertson, Wright and Dykstra}]{Robertson1988}
\textsc{Robertson, T.}, \textsc{Wright, F.~T.} and \textsc{Dykstra, R.~L.}
  (1988).
\newblock \textit{Order Restricted Statistical Inference}.
\newblock Wiley.

\bibitem[{Wright(1984)}]{Wright_1984}
\textsc{Wright, F.~T.} (1984).
\newblock The asymptotic behavior of monotone percentile regression estimates.
\newblock \textit{Canad. J. Statist.} \textbf{12} 229--236.

\end{thebibliography}

\end{document}